\documentclass[reqno]{amsart}

\usepackage[a4paper]{geometry}
\usepackage[T1]{fontenc}
\usepackage[utf8]{inputenc}
\usepackage{amssymb}
\usepackage[mathscr]{eucal}

\usepackage{mathtools}
\usepackage{enumerate}
\usepackage{accents}
\usepackage{comment}
\usepackage{dsfont}

\usepackage{color}
\usepackage[colorlinks=true]{hyperref}
\hypersetup{urlcolor=blue,citecolor=red,linkcolor=blue}
\usepackage[initials]{amsrefs}

\newtheorem{theo}{Theorem}[section]
\newtheorem{definition}[theo]{Definition}
\newtheorem{example}[theo]{Example}
\newtheorem{lem}[theo]{Lemma}
\newtheorem{propos}[theo]{Proposition}

\newtheorem{corollary}[theo]{Corollary}
\newtheorem{remark}[theo]{Remark}

\newcommand{\R}{{\mathbb{R}}}

\newcommand{\N}{{\mathbb{N}}}

\newcommand{\Z}{{\mathbb{Z}}}

\newcommand{\loglike}[1]{\mathop{\rm #1}\nolimits}

\newif\ifshort

\newcommand{\spn}{\loglike{Lin}}
\newcommand{\cspn}{\overline{\loglike{Lin}}}

\newcommand{\vertiii}[1]{{\left\vert\kern-0.25ex\left\vert\kern-0.25ex\left\vert #1
    \right\vert\kern-0.25ex\right\vert\kern-0.25ex\right\vert}}
    \newcommand{\vertiiis}[1]{{\vert\kern-0.25ex\vert\kern-0.25ex\vert #1
    \vert\kern-0.25ex\vert\kern-0.25ex\vert}}
\newcommand{\scal}[1]{{\left\langle\kern-0.25ex\left\langle #1
    \right\rangle\kern-0.25ex\right\rangle}}

\newcommand{\bea}{\begin{eqnarray*}}
\newcommand{\eea}{\end{eqnarray*}}
\newcommand{\beq}{\begin{eqnarray}}
\newcommand{\eeq}{\end{eqnarray}}

 \renewcommand{\le}{\leqslant}
\renewcommand{\leq}{\leqslant}
\renewcommand{\ge}{\geqslant}
\renewcommand{\geq}{\geqslant}

\numberwithin{equation}{section}
\begin{document}
\title[Linear Plasticity of Ellipsoids]{Linear Expand-Contract Plasticity of Ellipsoids Revisited}
\author[I.~Karpenko]{Iryna Karpenko}
\address{B.~Verkin Institute for Low Temperature Physics and Engineering,
47 Nauky Avenue, 61103 Kharkiv, Ukraine and
Universität Wien,
Oskar-Morgenstern-Platz 1
1090 Wien, Austria}
\email{iryna.karpenko@univie.ac.at}

\author{Olesia Zavarzina}
\address{Department of Mathematics and Informatics, V. N. Karazin Kharkiv National
University, 61022 Kharkiv, Ukraine.}
\email{olesia.zavarzina@yahoo.com}

\begin{abstract} This work is aimed to describe linear expand-contract plastic ellipsoids given via quadratic form of a bounded positively defined self-adjoint operator in terms of its spectrum.

\end{abstract}
\subjclass[2020]{46B20, 54E15}
\keywords{non-expansive map; ellipsoid; linearly expand-contract plastic space}
\maketitle

\section{Introduction}
Let $M$ be a metric space and $F\colon M\to M$ be a map. $F$ is called \emph{non-expansive} if it does not increase distance between points of the space $M$. $M$ is called \emph{expand-contract plastic} (or just plastic for short) if every non-expansive bijection $F\colon M \to M$ is an isometry.

There is a number of relatively recent publications devoted to plasticity of the unit balls of Banach spaces (see \cite{AnKaZa, CKOW2016, KZ, KZ2017, Zav}).
Here we give only one theorem which is a simple consequence of Theorem 1 in \cite{NaiPioWing} or Theorem 3.8 in \cite{Zav}.
\begin{theo}
Let $X$ be a finite-dimensional Banach space. Then $B_X$ is plastic.
\end{theo}
However, the question about plasticity of the unit ball of an arbitrary infinite-dimensional Banach space is open. At least, there are no counterexamples. On the other hand, an example of non-plastic ellipsoid in separable Hilbert space was built in \cite{CKOW2016}. In \cite{Zav1}, this example was generalized and the following definition was introduced.

\begin{definition}
 \emph{Let $M$ be a subset of a normed space $X$. We say that $M$ is} linearly expand-contract plastic \emph{(briefly an LEC-plastic) if every linear operator $T: X \to X$ whose restriction on $M$ is a non-expasive bijection from $M$ onto $M$ is an isometry on $M$}.
 \end{definition}

In the mentioned article \cite{Zav1} the ellipsoids of the following form were considered
  \[
  E =\left\{x = \sum_{n \in \N}x_n e_n  \in H: \sum_{n \in \N}\left|\frac{x_n}{a(n)}\right|^2 \le 1 \right\},
 \]
where $H$ is a separable Hilbert space with basis $\{e_n\}_1^\infty$ and $a(n)>0$. There was given a description of the LEC-plastic ellipsoids of such form. 

In what follows, we use the notations from \cite{Kad}. The letter $H$ denotes a fixed separable infinite-dimensional Hilbert space (real or complex), the symbol $\left\langle x, y \right\rangle$ stays for the scalar product of elements $x, y \in H$. We use the symbol $\spn$ to denote the linear span, and the symbol $\cspn$ to denote the closed linear span.

In the present paper, we will consider a more general definition of an ellipsoid, namely, 

\begin{definition} An ellipsoid in $H$ is a set of the form
 $$
  E =\left\{x \in H: \left\langle x, Ax \right\rangle \le 1 \right\},
 $$
where $A$ is a self-adjoint operator such that  $\inf_{x\in S_{H}} \left\langle Ax,x\right\rangle >0$ and $\sup_{x\in S_{H}} \left\langle Ax,x\right\rangle < \infty$.
\end{definition}

We will denote the boundary of $E$ by $$S =\left\{x \in H: \left\langle x, Ax \right\rangle = 1 \right\}.$$ 

In what follows, $\sigma(A)$ will stand for the spectrum of $A$. Note that in this case $\sigma(A)$ is bounded from below and above by some positive constants.

In this paper we will show, that in fact the description of LEC-plastic ellipsoids in \cite{Zav1} was complete. In other words, there is no other LEC-plastic ellipsoids, except for those already described.

 \section{Basic facts}

For our purpose, we will need some results related to measure theory  (see, e.g., \cite{Kad}, \cite{Teschl_RA},
\cite{MES}). Let us collect these results. 

Recall that a distribution function of a given positive finite Borel measure on the real numbers $\mu$ is given by
\[F_{\mu}(t)=\mu([0,t]).\]
Notice that this function is non-decreasing, and hence its generalized inverse $F^{-1}_{\mu}(t)=\sup\{x: F_\mu(x)\leq t\}$ is well-defined and also non-decreasing. 

Notice that any normalized atomless Borel measure $\mu$ can be mapped into Lebesgue measure $\lambda$ on $[0,1]$ (Indeed, performing a simple computation, we obtain $\mu([0,F^{-1}_{\mu}([0,t])])=t=\lambda ([0,t])$, $t\in[0,1]$). Hence, we get the following theorem.

\begin{theo}\label{thm:1} (see: \cite{MES}  or Theorem 9.2.2 in \cite{Bogachev})
All atomless standard probability spaces are mutually almost isomorphic.
\end{theo}

\begin{corollary}\label{Cor}
Let $\mu$ and $\nu$ be (finite and compactly supported) atomless Borel measures on $\mathbb{R}$ with $M_\nu=\nu(\mathbb{R})$ and $M_\mu=\mu(\mathbb{R})$. Then there exists a map $G_{\mu,\nu}$ such that $\nu = \frac{M_\nu}{M_\mu} \mu \circ G_{\mu,\nu}$.
\end{corollary}

\begin{remark} Observe that 
\begin{enumerate}
    \item  $G_{\mu,\nu}$ can be written explicitly in terms of corresponding distribution functions, namely, $G_{\mu,\nu}=F^{-1}_{\mu}\circ \frac{M_\mu}{M_\nu}F_{\nu}$.
    \item $G_{\mu,\nu}: supp(\nu)\longrightarrow supp(\mu)$. 
\end{enumerate}
\end{remark}

The next theorem can be found in \cite{Teschl_RA}. 

\begin{theo}\label{thm:2} (see: Theorem 2.16 in \cite{Teschl_RA}) Let $\mu$ be a given measure on $X$, $f:X\longrightarrow Y$ be measurable w.r.t. $\mu$ and $f_*\mu(A)$ be a measure on Y defined by $f_*\mu(A)=\mu(f^{-1}(A))$. Let $g:Y\longrightarrow \mathbb{C}$ be a Borel function. Then the function $g\circ f:X\longrightarrow \mathbb{C}$ is integrable w.r.t. $\mu$ if and only if $g$ is integrable w.r.t $f_*\mu$. Moreover
\[
\int_Y g d(f_*\mu)=\int_Y g\circ f d\mu.
\]
\end{theo}

Furthermore, we will need some results related to operator theory (see, e.g., \cite{ReedSimon}, \cite{Teschl}).

Let $A$ be a bounded self-adjoint operator on a separable Hilbert space $H$. 

Then we can introduce continuous functions of $A$, as it is shown in the following theorem:

\begin{theo}\label{Th_cont_func} (see: Theorem VII.1 in \cite{ReedSimon}, cf: Theorem 3.1. in \cite{Teschl})
Let $A$ be a bounded self-adjoint operator on a  Hilbert space $H$. Then there is a unique map $\phi_A:C(\sigma(A))\rightarrow \mathcal{L}(H)$ with the following properties:
\begin{enumerate}[(i)]
    \item $\phi_A(f g)=\phi_A(f)\phi_A(g)$, $\phi_A(\lambda f)=\lambda\phi_A(f)$,  $\phi_A(1)=I$, $\phi_A(\bar f)=\phi_A(f)^*$;
    \item $||\phi_A(f)||_{\mathcal{L}(H)}\leq C ||f||_\infty$;
    \item if $f(x)=x$, then $\phi_A(f)=A$.
\end{enumerate}
Moreover,
\begin{enumerate}[(iv)]
    \item if $A\psi=\lambda\psi$, than $\phi_A(f)\psi=f(\lambda)\psi$;
    \item $\sigma(\phi_A(f))=\{f(\lambda)|\lambda\in\sigma(A)\}$;
    \item if $f\geq 0$, then $\phi_A(f)\geq 0$;
    \item $||\phi_A(f)||=||f||_\infty$.
\end{enumerate}
Then we define $f(A):=\phi_A(f)$. 
\end{theo}

For every $\psi\in H$ we can define a corresponding linear functional on $C(\sigma(A))$ mapping $f\rightarrow \left\langle\psi, f(A)\psi\ \right\rangle$. Then by Riesz theorem, there exist a unique measure $\mu_\psi$ on $\sigma(A)$ such that $\left\langle\psi, f(A)\psi\right\rangle=\int_{\sigma(A)}f(\lambda)d\mu_\psi$. The measure $\mu_\psi$ is called the spectral measure associated with the vector $\psi$.

The next important result (spectral theorem) tells that every bounded self-adjoin operator can be realized as multiplication operator on a suitable measure space.

\begin{theo}\label{Spectral_Th} (see: Theorem VII.3 in \cite{ReedSimon}, or Lemma 3.4 and  Theorem 3.6 in \cite{Teschl})
Let $A$ be a bounded self-adjoint operator on a separable Hilbert space $H$. Then, there exist measures $\{\mu_n\}_{n=1}^N (N\in \N \text{ or } N=\infty)$ on $\sigma(A)$ and a unitary operator 
$$U\colon H \to \bigoplus_{n=1}^{N}L^2(\R,d\mu_n)$$
so that
$$(UAU^{-1}\psi)_n(\lambda)=\lambda\psi_n(\lambda)$$
where we write an element $\psi\in \oplus_{n=1}^{N}L^2(\R,d\mu_n)$ as an $N$-tuple $(\psi_1(\lambda),...,\psi_N(\lambda)).$ This realization of $A$ is called a spectral representation.
\end{theo}

Let us define $H_{pp}=\{\psi\in H|\mu_{\psi}$ is pure point$\}$, $H_{ac}=\{\psi\in H|\mu_{\psi}$ is absolutely continuous$\}$, $H_{sc}=\{\psi\in H|\mu_{\psi}$ is singularly continuous$\}$.

\begin{theo}\label{Th_decomp} (see: Theorem VII.4 in \cite{ReedSimon} or Lemma 3.19 in \cite{Teschl})
$H=H_{pp}\oplus H_{ac}\oplus H_{sc}$. Each of these subspaces is invariant under $A$. $A|_{H_{pp}}$ has a complete set of eigenvectors, $A|_{H_{ac}}$ has only absolutely continuous spectral measures and $A|_{H_{sc}}$ has only singularly continuous spectral measures.
\end{theo}

We will use the following notations:

\[\sigma_{pp}=\sigma(A|_{H_{pp}}),\]
\[\sigma_{cont}=\sigma(A|_{H_{cont}}),\text{ where } H_{cont}=H_{ac}\oplus H_{sc},\]
\[\sigma_{ac}=\sigma(A|_{H_{ac}}),\]
\[\sigma_{sc}=\sigma(A|_{H_{sc}}),\]
\[\sigma_{p}=\{\lambda| \lambda \text{ is an eigenvalue of } A\}.\]

Note that

\[\sigma_{cont}=\sigma_{ac}\cup\sigma_{sc},\]
\[\sigma(A)=\overline{\sigma_{p}}\cup\sigma_{cont}.\]

The following useful results can be found in \cite{Teschl} and \cite{Teschl_PDE}.

\begin{theo}\label{thm 10} (see: Theorem 2.20  in \cite{Teschl}) Let A be bounded self-adjoint. Then 
\[\inf\{\sigma(A)\}=\inf_{||x||=1}\left\langle x,Ax\right\rangle\]
and
\[\sup\{\sigma(A)\}=\sup_{||x||=1}\left\langle x,Ax\right\rangle\]
 
\end{theo}

Note that 
\[
\inf_{||x||=1}\left\langle x,Ax\right\rangle=\inf_{x\in H,~x\neq 0} \frac{\left\langle  x,Ax\right\rangle}{||x||^2},
\]
\[
\sup_{||x||=1}\left\langle x,Ax\right\rangle=\sup_{x\in H,~x\neq 0} \frac{\left\langle x,Ax\right\rangle}{||x||^2}.
\]
Moreover, one can show the following:

\begin{theo}\label{thm 11} (see Problem 13.1 in \cite{Teschl_PDE}) Let A be bounded self-adjoint (particularly, $\sigma(A)\subset [a,b]$). Then $\lambda_0:=\inf\{\sigma(A)\}$ is an eigenvalue iff $\inf_{||x||=1}\left\langle x,Ax\right\rangle$ is a minimum. In this case, eigenvectors are precisely the minimizers.  
\end{theo}

\section {Main result}
\begin{propos}\label{main_ex}
Suppose the spectrum $\sigma(A)$ of the self-adjoint operator $A$ contains a set of eigenvalues $B$ possessing the following properties:
\begin{enumerate}
\item $B$ has at least two elements;
\item either $B$ doesn't have minimum or the multiplicity of the minimum is infinite;
\item  either $B$ doesn't have maximum or the multiplicity of the maximum is infinite.
\end{enumerate}
Then $E$ is not LEC-plastic.
\end{propos}
\begin{proof}
Denote $r = \inf B$, $R = \sup B$; according to (1) $r < R$. The property (2) ensures the existence of distinct $n_k \in \N$, $k = 1, 2, \ldots$ such that eigenvalues $\lambda_{n_k} \in B$, $\lambda_{n_k} < \frac12 (r + R)$ and
$$
\lambda_{n_1} \ge \lambda_{n_2} \ge \lambda_{n_3} \ge \ldots, \quad \lim_{k \to \infty} \lambda_{n_k}  = r.
$$
Analogously, the property (3) gives us the existence of distinct $n_k \in \N$, $k = 0, -1, -2, \ldots$ such that $\lambda_{n_k}\in B$ and
$$
\lambda_{n_1} < \lambda_{n_0} \le \lambda_{n_{{-1}}} \le \lambda_{n_{-2}} \le \ldots, \quad \lim_{k \to - \infty} \lambda_{n_k}  = R.
$$
Take in $H$ the orthonormalized eigenvectors $e_{n_k}$ corresponding to $\lambda_{n_k}$ (in case of infinite multiplicity of the minimum we are choosing them to be ONB of $Ker(A-\lambda_{min})$, analogously for the maximum) and extend to an orthonormal basis $e_n$ in $H$.
Define the linear operator $T$ as follows: $Te_n = e_n$ for $n  \in \N \setminus \{n_k\}_{k \in \Z}$, and $Te_{n_k} = \sqrt{\frac{\lambda_{n_{k}}}{\lambda_{n_{k-1}}}}e_{n_{k-1}}$, for $k \in \Z$. Using the fact that for a self-adjoint operator $A$ and $x=\sum x_k e_{n_k}$ (here $e_{n_k}$ are eigenvectors of $A$): $\left\langle x+x^{\perp}, A(x+x^{\perp})\right\rangle=\left\langle x,Ax\right\rangle + \left\langle x^{\perp},Ax^{\perp}\right\rangle$, we get that the linear non-expansive operator $T$ maps $E$ onto itself bijectively but not isometrically.
\end{proof}

Now, let us consider the case when the operator has purely continuous spectrum. 

\begin{propos}\label{prop:2.7.}
Let $\mu(t)$ be an (finite) atomless Borel measure with $supp(\mu)\subset(\alpha,\beta)$, $\alpha>0$. Let $A\colon L_2(\mathbb{R},d\mu(t))\to L_2(\mathbb{R},d\mu(t))$ be an operator acting by the rule
$$Af(t)=tf(t).$$
Consider an ellipsoid $E\in L_2(\mathbb{R},d\mu(t))$ generated by this operator $A$, i.e.
$$E=\{f\in L_2(\mathbb{R},\mu(t)dt): \left\langle Af,f\right\rangle = \int_{\alpha}^{\beta} t|f(t)|^2d\mu(t)\leq 1 \}.$$
$E$ is not LEC-plastic.
\end{propos}

\begin{proof} Note that under our assumptions $L_2(\mathbb{R},d\mu(t))$ can be identified with
$L_2((\alpha,\beta),d\mu(t))$.

Consider some infinite partition of the segment
$$[\alpha,\beta) =\bigsqcup_{k=-\infty}^{\infty}\Delta_k,$$
where $\Delta_k=[a_k, a_{k+1})$, $a_{-\infty}=\alpha$, $a_{\infty}=\beta$, and such that $\mu(\Delta_k)>0$ for all $k$ (we can do this e.g. using any convergent series). Note that $L_2((\alpha,\beta),d\mu(t))=\oplus_2 L_2(\Delta_k,d\mu_k(t))$ with $\mu_k:=\mu|_{\Delta_k}$.
We will use the following notation $M_{\mu_k}:=\mu_k(\Delta_k)=\mu(\Delta_k)$.
By Theorem \ref{thm:1}, the spaces $(\Delta_k,\mu_k(t))$ are mutually almost isomorphic and we will denote the corresponding isomorphism by $G_k:=G_{\mu_k,\mu_{k+1}}\colon supp(\mu_{k+1})\longrightarrow supp(\mu_{k})$ .

Introduce the associated operators $H_k\colon L^2(\Delta_k,d\mu_k)\to L^2(\Delta_{k+1},d\mu_{k+1})$ acting as 
\[
H_kf_k=f_k\circ G_k\cdot\sqrt{\frac{M_{\mu_{k+1}}}{M_{\mu_{k}}}}.
\]


Then using Theorem \ref{thm:2} in the first step and Corollary \ref{Cor} on the second step, we get
\begin{equation*}
\int_{\Delta_{k+1}}\left|f_k\circ G_k\right|^2\frac{M_{\mu_{k+1}}}{M_{\mu_{k}}}d\mu_{k+1}=
\int_{\Delta_{k}}|f_k|^2\frac{M_{\mu_{k+1}}}{M_{\mu_{k}}}d(G_{k*}\mu_{k+1})=\int_{\Delta_k}|f_k|^2d\mu_k.
\end{equation*}
For the second step note that $G_{k*}\mu_{k+1}=\mu_{k+1}\circ G_k=\frac{M_{\mu_{k}}}{M_{\mu_{k+1}}}\mu_k$. This means that \[\|H_kf_k\|_{L^2(\Delta_{k+1},d\mu_{k+1})}=\|f_k\|_{L^2(\Delta_{k},d\mu_{k})}.\]
Now, consider an operator $T\colon L_2((\alpha,\beta),d\mu(t))\to L_2((\alpha,\beta),d\mu(t))$ acting as
$$Tf_k=g_kH_kf_k,$$
where $g_k$ will be defined later.

Then we have
$$\left\langle f,Af\right\rangle =\sum_{k\in \Z}\int_{\Delta_k}t|f(t)|^2d\mu_k(t),$$
and using again Theorem \ref{thm:2} and Corollary \ref{Cor}, we get
\begin{align*}
\left\langle Tf,ATf\right\rangle &=\sum_{k\in \Z}\int_{\Delta_{k+1}}t|g_k(t)|^2\left|(f_k\circ G_k)(t)\right|^2\frac{M_{\mu_{k+1}}}{M_{\mu_{k}}}d\mu_{k+1}(t)\\
&=\sum_{k\in \Z}\int_{\Delta_{k}}G_k^{-1}(s)\left|(g_k\circ G_k^{-1})(s)\right|^2|f_k(s)|^2d\mu(s).
\end{align*}
Let us denote $\hat g_k(s):=g_k G_k^{-1}(s)$.

We will choose $g_k$ ($\hat g_k(s)$) in such way that $\left\langle f,Af\right\rangle=\left\langle Tf,ATf\right\rangle$ (then $T:E\longrightarrow E$ is bijective), i.e.
\[\hat g_k(s):=\sqrt{\frac{s}{G_k^{-1}(s)}},\quad s\in \Delta_k.\]

It remains to show that $T$ is non-expansive but not an isometry. Notice that $\hat g_k(s)^2\leq 1$ implies that $T$ is non-expansive. Then to show that $T$ is not an isometry it suffices to show that there is $k\in \N$ and $t\in \Delta_k$ such that $\hat g_k(s)^2< 1$ in some neighbourhood of this $t$.
Indeed, for $s\in\Delta_k$ by construction $G_k^{-1}(s)\in\Delta_{k+1}$, which means $\hat g_k(s)^2< 1$ for $s\in\Delta_k$. Thus $T$ is non-expansive but not an isometry.
\end{proof}

\begin{corollary} Let  $A$ be a bounded self-adjoint operator.
Let $\sigma_{cont}\neq \{\emptyset\}$. Then the ellipsoid generated by this operator is not LEC-plastic.
\end{corollary}
\begin{proof}
Indeed, due to the Theorem \ref{Th_decomp} we can split an operator into two pieces: the first piece corresponds to the continuous measure, the second piece corresponds to the rest.  Then we can apply Theorem \ref{Spectral_Th}  to the first piece and get that there is a spectral representation of $A|_{\sigma_{cont}}$. One may chose one of those parts where $A$ acts as multiplication on the independent variable and the construction described in the previous proposition allows us to obtain an operator $T$ which is non-expansive bijection, but not an isometry.
\end{proof}

The next three lemmas are, in fact, building blocks for the proof of Theorem \ref{main1}.  

\begin{lem}\label{lem}
Let $T\colon H \to H$ be a linear operator which maps $E$ bijectively onto itself. Then $T$ maps the whole of $H$ bijectively onto itself and  $T(S) = S$. If, moreover, $T$ is non-expansive on $E$, then $\|T\| \le 1$.
\end{lem}
\begin{proof}
$E$ is absorbing since it contains a ball (the spectrum is bounded from below by positive constant). 

The rest of the proof is as in the article \cite{Zav1}. For convenience of the reader we give it here.

$E$ is an absorbing set, so $H = \cup_{t > 0}tE$. By linearity $T$ is injective on every set $tE$, consequently it is injective on the whole $H$. Also, $T(H) = \cup_{t > 0}T(tE) = \cup_{t > 0}tE = H$ which gives the surjectivity on $H$. Finally, $S = E \setminus \cup_{t \in (0, 1)}tE$, so $T(S) = T(E) \setminus \cup_{t \in (0, 1)}T(tE) = E \setminus \cup_{t \in (0, 1)}tE = S$.  If, moreover, $T$ is non-expansive on $E$, then for every $x \in H$ there is a $t > 0$ such that $tx \in E$ and we have $\|T(tx)\| = \rho(T(0), T(tx)) \le \rho(0, tx) = \|tx\|$. It remains to divide by $t$ to obtain that $\|Tx\| \le \|x\|$ for all  $x \in H$.

\end{proof}

Before moving to the next result, let us introduce the following notation:
$$H_t=Ker(A-t).$$

\begin{lem}\label{lem_ind}
 Let $\sigma(A)=\sigma_{pp}(A)\subset (0,+\infty)$ and let the set of eigenvalues of an operator $A$ contain the minimal element $r$  and let $r$ have finite multiplicity. Let $T\colon H \to H$ be a linear operator  which maps $E$ bijectively onto itself and whose restriction on $E$ is non-expansive. Then $T(H_r) = H_r$, $T\left(H_r \cap E \right) = H_r \cap E$ and the restriction of $T$ onto $H_r$ is a bijective isometry.
\end{lem}
\begin{proof}
Theorem \ref{thm 10} implies that $r=\inf_{||x||=1}\left\langle     x,Ax\right\rangle=\inf_{x\in H,~x\neq 0} \frac{\left\langle x,Ax\right\rangle}{||x||^2}$.

Recall that for $x\in S$, we have $\left\langle x,Ax\right\rangle=1$. Hence, $r=\inf_{x\in H,~x\neq 0} \frac{\left\langle x,Ax\right\rangle}{||x||^2}=\inf_{x\in S} \frac{\left\langle x,Ax\right\rangle}{||x||^2}=\inf_{x\in S} \frac{1}{||x||^2}$ (second equality: renormalization using $\left\langle x,Ax\right\rangle$ instead of norm). That is, for $x\in S$ we have $||x||\leq \sqrt{\frac{1}{r}}$. 

Let $x\in S$ with $||x||= \sqrt{\frac{1}{r}}$ (i.e. by Theorem \ref{thm 11} it is an eigenvector corresponding to $r$. Moreover, all eigenvectors can be written in this form after a suitable renormalisation). Since $T$ is non-expansive on $E$, we have $||T^{-1}(x)||\geq ||x||=\sqrt{\frac{1}{r}}$ and using  $T^{-1}(S)=S$, we get $||T^{-1}(x)||=\sqrt{\frac{1}{r}}$. Hence by Theorem \ref{thm 11} $T^{-1}(x)$ is an eigenvector corresponding to $r$. Hence, $T^{-1}(H_r)\subset H_r$.

Note that $H_r$ is finite dimensional, hence $T^{-1}|_{H_r}:H_r\mapsto H_r$ is surjective iff $T^{-1}:H_r\mapsto H_r$ is injective.  \ifshort\else Injectivity of $T^{-1}|_{H_r}$ follows from Lemma \ref{lem} (i.e. injectivity of $T^{-1}$). \fi  
Hence $T^{-1}|_{H_r}:H_r\mapsto H_r$ is bijection and $T^{-1}(H_r)= H_r$ (i.e. $T(H_r)= H_r$). Hence also $T(H_r\cap E)= H_r\cap E$ (combining $T(H_r)= H_r$ with $T(E)=E$).

The last claim we obtain using linearity and the fact that for all eigenvectors $x$ corresponding to $r$ with norm $\sqrt{\frac{1}{r}}$, we have that  $||T^{-1}(x)||=\sqrt{\frac{1}{r}}$. (We can also observe that $S\cap H_r$ is a sphere in $H_r$ and map $T$ from $S\cap H_r$ onto $S\cap H_r$ is bijective.)

\end{proof}

\begin{lem}\label{lem_ind2}
 Let $\sigma(A)=\sigma_{pp}(A)\subset (0,+\infty)$ and let the set of eigenvalues of an operator $A$ contain the maximal element $R$  and let $R$ have finite multiplicity. Let $T\colon H \to H$ be a linear operator  which maps $E$ bijectively onto itself and whose restriction on $E$ is non-expansive. Then $T(H_R) = H_R$, $T\left(H_R \cap E \right) = H_R \cap E$ and the restriction of $T$ onto $H_R$ is a bijective isometry.
\end{lem}
\begin{proof}
The proof is similar to the previous one (alternatively consider $T^{-1}$ instead of $T$ and apply the previous result).

\end{proof}

The proof of the next theorem partially repeats the proof of Theorem 1 in \cite{Zav1}. To make this work self-contained we provide all the details here.

\begin{theo}\label{main1}
Let  $A$ be a bounded self-adjoint operator. Then an ellipsoid $E$ generated by $A$ is LEC-plastic if and only if the following two conditions hold:
\begin{enumerate}
    \item $\sigma_{cont}=\emptyset$;
    \item every subset of $\sigma_p(A)$ that consists of more than one element  either has a maximum of finite multiplicity or has a minimum of finite multiplicity.
\end{enumerate}

\end{theo}

\begin{proof} We only need to prove the "if" part of the statement. 
Note that under our assumptions the spectrum $\sigma(A)=\overline{\sigma_{p}(A)}=\sigma_{pp}(A)$ (i.e. consists of eigenvalues and their limiting points). Moreover $A$ cannot contain more than one element of infinite multiplicity.

Note that in this case there exists a basis of eigenvectors of $A$ (see Problem 3.26 in \cite{Teschl}).

\textit{Claim 1.} There is a $\tau > 0$ such that $A^+ =\sigma_{p}(A)\cap(\tau, +\infty)$ is well-ordered with respect to the
ordering $\geq$ (that is every not empty subset of $A^+$ has a maximal element), $A^{-}=\sigma_{p}(A)\cap(0,\tau)$
is well-ordered with respect to the ordering $\leq$ (that is every not empty subset of $A^-$ has a
minimal element), and neither $A^+$ nor $A^-$ contain elements of infinite multiplicity.

Indeed, if there is an element $a_\infty\in \sigma_{p}(A)$ of infinite multiplicity, let us take $\tau=a_\infty$. Let
us demonstrate that $(A^+, \geq)$ is well-ordered. If $A^+ = \emptyset$ the statement is clear. In the other
case for every not empty subset $D$ of $A^+$ consider $B = \{\tau\} \cup D$. Then the minimal element
of $B$ is $\tau$ , which has infinite multiplicity so $B$ must have a maximum of finite multiplicity.
This maximum will be also the maximal element of $D$. The demonstration of well ordering
for $(A^{-} , \leq)$ works in the same way.

Now, consider the remaining case of $\sigma_{p}(A)$ consisting only of finite multiplicity elements.
Consider the set $U$ of all those $t \in (0, +\infty)$ that $\sigma_{p}(A) \cap (t, +\infty)$ is not empty and well-ordered
with respect to the ordering $\geq$. If $U$ is not empty, take $\tau =\inf U$, if $U = \emptyset$, take $\tau = \sup \sigma_{p}(A)$.
Let us demonstrate that this $\tau$ is what we need. In the first case $A^+ = \sigma_{p}(A)\cap (\tau, +\infty)$ and for
every $t > \tau$ we have $\sigma_{p}(A) \cap (t, +\infty)$ is not empty and well-ordered with respect to the ordering $\geq$. This implies that $(A^{+}, \geq)$ is well-ordered. In the second case $A^+ = \emptyset$, which is also well-
ordered. So, it remains to demonstrate that $A^{-} = \sigma_{p}(A) \cap (0, \tau )$ is well-ordered with respect to the ordering $\leq$. Assume this is not true. Then, there is a non empty subset $B \subset A^{-}$ with no
minimal element. According to the conditions of our theorem $B$ has a maximal element $ b$.
Since $b < \tau$ and by definition of $\tau$ the set $\sigma_{p}(A) \cap (b, +\infty)$ is not well-ordered with respect
to the ordering $\geq$. Consequently, there is a non empty $D \subset \sigma_{p}(A)\cap (b, +\infty)$ with no maximal
element. Then, $B \cup D $ satisfies neither condition (1) nor condition (2) of our theorem. This
contradiction completes the demonstration of Claim 1.

We introduce the following three subspaces:
\begin{itemize}
    \item $H^-$ is the closed linear span of the set of all those eigenvectors, for which the corresponding eigenvalue lies in $A^-$;
    
    \item $H^\tau$ is the closed linear span of the set of all those eigenvectors, for which the corresponding eigenvalue is $\tau$ if $\tau$ is eigenvalue or empty otherwise; 
        
    \item $H^+$ is the closed linear span of the set of all those eigenvectors, for which the corresponding eigenvalue lies in $A^+$.
\end{itemize}

Since eigenvectors of a self-adjoint operator corresponding to different eigenvalues are orthogonal, using continuity of scalar product, we get that these subspaces are mutually orthogonal and  $H=H^- \oplus H^\tau \oplus H^+$ (we have equality here since the set of all eigenvectors of $A$ spans $H$). Let $T:H\mapsto H$ be a linear operator which maps $E$
bijectively onto itself and whose restriction on E is non-expansive.

\ifshort
\else
\textbf{Claim 2}. \emph{$T(H^-) = H^-$, $T(H^+) = H^+$ and the restrictions of $T$ onto $H^-$ and $H^+$ are bijective isometries}.

We will demonstrate the part of our claim that speaks about $H^+$: the reasoning about $H^-$ will differ only in the usage of Lemma \ref{lem_ind} instead of  Lemma \ref{lem_ind2}.

Let us define a subspace $H(t)$ as the closed linear span of the set of all those eigenvectors, for which the corresponding eigenvalue lies in $A^+\cap[t,+\infty)$.

If $A^+ = \emptyset$ there is nothing to do. In the case of $A^+ \neq \emptyset$ we are going to demonstrate
by transfinite induction in $t \in (A_+, \geq)$ the validity for all $t \in A_+$ of the following statement $\frak U(t)$: the subspace $H(t)$ is $T$-invariant and $T $ maps $H(t)$ onto $H(t)$ isometrically. Since the collection of subspaces $H(t)$, $t \in A_+ $ is a chain whose union is dense
in $H_+$, the continuity of $T$ will imply the desired Claim 2.

The base of induction is the statement  $\frak U(t)$ for $t = \max A$. This is just the statement of Lemma \ref{lem_ind2}. We assume now as inductive hypothesis the validity of $\frak U(t)$ for all $t > t_0 \in A^+$, and our goal is to prove the statement $\frak U(t_0)$.
For every $x,y$ in $H$ let us introduce a modified scalar product $\scal{x,y}$ as follows:
$$
\scal{x,y}=\left\langle x,Ay\right\rangle.
$$
Then the norm on $H$ induced by this modified scalar product is
$$
\vertiiis{x}=\sqrt{\scal{x,x}}=\sqrt{\left\langle x,Ax\right\rangle}.
$$
The ellipsoid $E$ is the unit ball in this new norm and since $T$ is linear and maps $E$ onto $E$ bijectively, $T$ is a bijective isometry of $\left(H,  \vertiiis{\cdot}\right)$ onto itself.
Due to \cite[Theorem 2, p. 353]{Kad} $T$ is a unitary operator in the modified scalar product and thus $T$ preserves the modified scalar product. In particular, it preserves the orthogonality in the modified scalar product.
Denote 

\ifshort
$$X = \bigcup_{t > t_0}H(t) = \spn\{e_k|e_k \text{ are the eigenvectors corresponding to } \lambda_k\in A^+\cap(t_0,+\infty)\}.$$

The orthogonal complement to $X$ in the modified scalar product is 
$$X^\perp =  \overline{\spn}\{e_k|e_k \text{ are the eigenvectors corresponding to } \lambda_k\in A^+\cap(0,t_0]\}.$$

\else
$$X = \bigcup_{t > t_0}H(t).$$

In other words, $X$ is the linear span of the set of all those eigenvectors, for which the corresponding eigenvalue lies in $A^+\cap(t_0,+\infty)$.

The orthogonal complement to $X$ in the modified scalar product $X^\perp$ is the closed linear span of the set of all those eigenvectors, for which the corresponding eigenvalue lies in $A^+\cap(0,t_0].$
\fi

Occasionally the orthogonal complement to $X$ in the original scalar product is the same.  Our  inductive hypothesis implies that $T(X) = X$, consequently $T(X^\perp) = X^\perp$ and  $T(X^\perp \cap E) = X^\perp \cap E$.

$X^\perp$ equipped with the original scalar product is a Hilbert space, $X^\perp \cap E$ is an ellipsoid in $X^\perp$, $t_0$ is the maximal eigenvalue of $A$  and the multiplicity of  $t_0$ is finite because  $t_0 \in A^+$. The application of Lemma \ref{lem_ind2} gives us that $T(H_{t_0}) = H_{t_0}$  and the restriction of $T$ onto $H_{t_0}$ is a bijective isometry in the original norm. Now, $T$ maps $X$ onto $X$ isometrically, maps $H_{t_0}$  onto $H_{t_0}$ isometrically and $H(t_0)$ is the orthogonal direct sum of subspaces $H_{t_0}$ and the closure of $X$. This implies that $T$ maps $H(t_0)$ onto $H(t_0)$ isometrically, and the inductive step is done. This completes the demonstration of Claim 2.

From Claim 2 and mutual orthogonality of $H^-$ and $H^+$ we deduce that  $T(H^- \oplus H^+) = H^- \oplus H^+$ and $T$ is an isometry on  $H^- \oplus H^+$. Recalling again that $T$ preserves the modified scalar product and the fact that the orthogonal complement to $X$ in the modified scalar product is $H_\tau$ we obtain that $T(H_\tau) = H_\tau$  and consequently $T(H_\tau \cap E) = H_\tau \cap E$. But $H_\tau \cap E$ is equal to the closed ball (in the original norm) of radius $\frac{1}{\sqrt{\tau}}$ centered at 0, so the equality $T(H_\tau \cap E) = H_\tau \cap E$ and linearity of $T$ implies that $T$ is an isometry on $H_\tau$. Finally, as we know, $H = H^- \oplus H_\tau \oplus H^+$, so $T$ is an isometry on the whole $H$.
\fi
\end{proof}

\textbf{Acknowledgement}. The authors are grateful to Vladimir
Kadets and Gerald Teschl for constant support and useful advices. The second author was supported  by the  National Research Foundation of Ukraine funded by Ukrainian State budget in frames of the project 2020.02/0096 ``Operators in infinite-dimensional spaces:  the interplay between geometry, algebra and topology''.

\bibliographystyle{amsplain}

\begin{thebibliography}{10}


\bibitem{AnKaZa} Angosto C., Kadets V., Zavarzina O. \textit{Non-expansive bijections, uniformities and polyhedral faces}, J. Math. Anal. Appl., \textbf{471} No. 1-2 (2019) 38-52. 

\bibitem{Bogachev} Bogachev V.I. \textit{Measure theory}, Springer-Verlag, 2007.

\bibitem{CKOW2016} Cascales B.,  Kadets V., Orihuela J.,  Wingler E.J. \textit{Plasticity of the unit ball of a strictly convex Banach space}, Revista de la Real Academia de Ciencias Exactas, F\'{\i}sicas y Naturales. Serie A. Matem\'aticas,  {\bf 110} No. 2 (2016) 723--727.

\bibitem{KZ} Kadets V.,  Zavarzina O.  \textit{Plasticity of the unit ball of $\ell_1$},  Visn. Hark. nac. univ. im. V.N. Karazina, Ser.: Mat. prikl. mat. meh.,  \textbf{83} (2017) 4--9.

\bibitem{Kad} Kadets V.   \textit{A course in Functional Analysis and Measure Theory}, Springer, 2018.

\bibitem{KZ2017} Kadets V.,  Zavarzina O.  \textit{Non-expansive bijections to the unit ball of $\ell_1$-sum of  strictly convex Banach spaces}, Bulletin of the Australian Mathematical Society, \textbf{97}, No. 2 (2018) 285--292.


\bibitem{NaiPioWing}  Naimpally S. A., Piotrowski Z., Wingler E. J.  \textit{Plasticity in metric spaces}, J. Math. Anal. Appl., \textbf{313} (2006) 38--48.

\bibitem{ReedSimon} Reed M., Simon B.\textit{Methods of modern mathematical physics I: Functional Analysis}, Academic Press, 1980.

\bibitem{Teschl_RA} Teschl G. \textit{Topics in Real Analysis}, Graduate Studies in Mathematics, to appear.

\bibitem{Teschl_PDE} Teschl G. \textit{Partial Differential Equations}, Graduate Studies in Mathematics, to appear.

\bibitem{Teschl} Teschl G. \textit{Mathematical Methods in Quantum Mechanics}, Graduate Studies in Mathematics, Volume 157, Amer. Math. Soc., Providence, 2014.

\bibitem{Zav} Zavarzina O. \textit{Non-expansive bijections between unit balls of Banach spaces},   Annals of Functional Analysis, \textbf{9}, No. 2 (2018) 271--281.

\bibitem{MES}
\url{https://encyclopediaofmath.org/wiki/Standard_probability_space}

\bibitem{Zav1}	Zavarzina O. \textit{Linear expand-contract plasticity of ellipsoids in separable Hilbert spaces}, Matematychni Studii, \textbf{51} No. 1 (2019) 86–91.

\end{thebibliography}
\end{document}